\documentclass[reqno,11pt]{amsart}

\usepackage{graphicx}

\usepackage[ansinew]{inputenc}
\usepackage{amsfonts,epsfig}
\usepackage{latexsym}
\usepackage{mathabx}
\usepackage{amsmath}
\usepackage{amssymb}

\usepackage{color}

\newtheorem{theorem}{Theorem}
\newtheorem{lemma}[theorem]{Lemma}

\theoremstyle{definition}

\theoremstyle{remark}

\numberwithin{equation}{section}

\newcommand{\D}{\mathbb{D}}
\newcommand{\DD}{\widehat{\mathcal{D}}}
\newcommand{\Dd}{\widecheck{\mathcal{D}}}

\newcommand{\DDD}{\mathcal{D}}

\newcommand{\RR}{\mathcal{R}}

\newcommand{\C}{\mathbb{C}}

\renewcommand{\phi}{\varphi}

\newcommand{\whw}{\widehat{\omega}}

\def\a{\alpha}               
           
     \def\om{\omega}      
              
                  \def\z{\zeta}

\renewcommand{\H}{\mathcal{H}}

\begin{document}
\title[Schatten Class Hankel Operators on Weighted Bergman Spaces]{Schatten Class Hankel Operators on Weighted Bergman Spaces induced by regular weights}

\today

\thanks{The second author is indebted to University of Eastern Finland for the financial support.}

\begin{abstract}
In this paper, for $1\leq p<\infty$, we provide several descriptions of Schatten $p$-class Hankel operators $H_f$ and $H_{\overline{f}}$ on the weighted Bergman space $A^2_\om$, in terms of a certain global and local mean oscillation of the symbol $f\in L^2_\om$, provided $\om$ is in a class of regular weights. The approaches applied to rely on several classical methods, and simultaneously rely on a novel but more convenient construction associated with the atomic decomposition of $A^2_\om$.
\end{abstract}

\keywords{Schatten class, Weighted Bergman space, Regular weight, Mean oscillation.}

\subjclass[2010]{Primary 30H20, 47B35}

\author[Hamzeh Keshavarzi]{Hamzeh Keshavarzi}
\address{Department of Mathematics, College of Sciences, Shiraz University, Shiraz, Iran}
\email{amzehkeshavarzi67@gmail.com}

\author[Fanglei Wu]{Fanglei Wu}
\address{University of Eastern Finland, P.O.Box 111, 80101 Joensuu, Finland}
\email{fanglei.wu@uef.fi}
\email{fangleiwu1992@gmail.com}

\maketitle

\section{Introduction and main results}

Let $\H(\D)$  denote the space of analytic functions in the unit disc $\D=\{z\in\C:|z|<1\}$. A non-negative function $\om\in L^1(\D)$ such that $\om(z)=\om(|z|)$ for all $z\in\D$ is called a radial weight. For $0<p<\infty$ and such an $\omega$, the Lebesgue space $L^p_\om$ consists of complex-valued measurable functions $f$ on $\D$ such that
    $$
    \|f\|_{L^p_\omega}^p=\int_\D|f(z)|^p\omega(z)\,dA(z)<\infty,
    $$
where $dA(z)=\frac{dx\,dy}{\pi}$ is the normalized area measure on $\D$. The corresponding weighted Bergman space is $A^p_\om=L^p_\omega\cap\H(\D)$. As usual, we write $A^p_\eta$ for the classical weighted Bergman spaces induced by the
standard weight $\om(z)=(\eta+1)(1-|z|^2)^\eta$ with $-1<\eta<\infty$.

For a radial weight $\om$, the norm convergence in $A^2_\om$ implies the uniform convergence on compact subsets, and therefore the Hilbert space $A^2_\om$ is a closed subspace of $L^2_\om$ and the orthogonal Bergman projection $P_\om:L^2_\om\to A^2_\om$ is given by
	$$	
	P_\om(g)(z)=\int_\D g(z)\overline{B^\om_z(\z)}\om(\z)\,dA(\z),\quad z\in\D,
	$$
where $B^\om_z$ is the reproducing kernel of $A^2_\om$, associated with the point $z\in\D$. In the case of a standard weight, the Bergman reproducing kernels are given by the neat
formula $(1-\overline{z}\z)^{-(2+\eta)}$. For the convenience, throughout the paper we assume $K^\eta_z(\z)=(1-\overline{z}\z)^{-\eta}$, and hence $K^{\eta+2}_z(\z)$ is the kernel of $A^2_\eta$. The most commonly known result on the Bergman projection is due to Bekoll\'e and Bonami \cite{bekolle1,bekolle2}, see \cite{pelaez2,PelRat2020} for recent results and the reference therein.

A compact operator $T$ from a Hilbert space $H$ to another Hilbert space $K$ is said to belong to the Schatten class $S_p=S_p(H, K)$ if its sequence of singular numbers $\{\lambda_n\}$ belongs to
$\ell_p$ ($0<p<\infty$). It is well-known that for $1\leq p<\infty$, the class $S_p$ is a Banach space with the norm $\|T\|_{S_p}=(\sum_j |\lambda_j|^p)^{1/p}$. Moreover, $S_p$ is closed under the product of operators, in other words, if  $T\in S_p$, $A$ is a bounded operator on $H$ and $B$ a bounded operator on $K$, then $BTA\in S_p$. See \cite{zhu2} for more basic information about the Schatten class.

One important linear operator related to the Bergman projection $P_\om$ is known as the (big) Hankel operator:
$$
H_f(g)(z)=(I-P_\om)(fg)(z),\quad f\in L^1_\om,~~z\in\D.
$$
The study of the Schatten class Hankel operators $H_f$ and $H_{\overline{f}}$ on Hilbert-weighted Bergman spaces is a compelling topic that has attracted considerable attention during the last decades. When the symbol $f$ is analytic, we refer to \cite{abate1,arazy,Janson1988,Wallsten1990,Zhu1990} for study of the Schatten class $H_f$ and $H_{\overline{f}}$ on standard weighted Bergman spaces $A^2_\eta$ in the unit ball $\mathbb{B}_n$ of $\C^n$. When $f$ is considered to be a general integrable function, for the case $2\leq p<\infty$, Zhu \cite{zhu3} characterized the Schatten class $H_f$ and $H_{\overline{f}}$ simultaneously on $A^2_\eta$ in $\mathbb{B}_n$ in terms of a certain local mean oscillation of the symbol $f$ associated with Bergman metric; using the same characterization, Xia \cite{xia1,xia2} and Isralowitz \cite{isralowitz} solved the question for the case $\max\{1, \frac{2n}{n+1+\eta}\}<p\leq2$ and $\frac{2n}{n+1+\eta}<p\leq1$ respectively; and finally Pau \cite{pau2} closed the full case $0<p<\infty$ in terms of the same local mean oscillation. Apart from such local mean oscillation, another tool concerning the problem is known as the so-called (global) mean oscillation related to a Berezin-type transform of the symbol $f$. As a matter of fact, from the aforementioned literature, these two characterizations of $H_f, H_{\overline{f}}\in S_p$ are equivalent.

The purpose of the paper is to characterize Schatten class $H_f$ and $H_{\overline{f}}$ on $A^2_\om$, provided $\om$ is a class of certain locally smooth weights. We now proceed toward the exact statements via necessary definitions. Throughout this paper we assume $\widehat{\om}(z)=\int_{|z|}^1\om(s)\,ds>0$ for all $z\in\D$, for otherwise $A^p_\om=\H(\D)$. A weight $\om$ belongs to the class~$\DD$ if there exists a constant $C=C(\om)\ge1$ such that $\widehat{\om}(r)\le C\widehat{\om}(\frac{1+r}{2})$ for all $0\le r<1$. Moreover, if there exist $K=K(\om)>1$ and $C=C(\om)>1$ such that $\widehat{\om}(r)\ge C\widehat{\om}\left(1-\frac{1-r}{K}\right)$ for all $0\le r<1$, then we write $\om\in\Dd$. In other words, $\om\in\Dd$ if there exists $K=K(\om)>1$ and $C'=C'(\om)>0$ such that
	\begin{equation}\label{Dcheck}
	   \widehat{\om}(r)\le C'\int_r^{1-\frac{1-r}{K}}\om(t)\,dt,\quad 0\le r<1.
	\end{equation}
The intersection $\DD\cap\Dd$ is denoted by $\DDD$. The class $\RR\subset\DDD$ of regular weights consists of those radial weights for which $\widehat{\om}(r)\asymp\omega(r)(1-r)$ for all $0\le r<1$. We immediately see that all standard weights belong to $\RR$. The true advantage of the class $\RR$ is the local smoothness of its weights. It is clear
that if $\om\in\RR$, then for each $s \in[0,1)$ there exists a constant $C=C(s, \omega)>1$ such that
$$
C^{-1} \omega(t) \leq \omega(r) \leq C \omega(t), \quad 0 \leq r \leq t \leq r+s(1-r)<1.
$$

It will turn out that the global mean oscillations defined by a certain Berezin-type transform and local mean oscillations related to the Bergman metric are both efficient tools for depicting $H_f, H_{\overline{f}}\in S_p$. Note that for any radial weight
$\om$, there exists a sufficiently large $\eta=\eta(\om)>0$ such that the inclusion $A^2_\om\subset A^2_{\eta}$, which makes $k^{\eta+2}_{\om,z}:={K^{\eta+2}_z}/{\|K^{\eta+2}_z\|_{A^2_\om}}$ well-defined. Now, for $g\in L^1_\om$ and such $\eta$, the Berezin-type transform  $B_{\omega,\eta}(g)$ is defined as
$$
B_{\omega,\eta}(g)(z)=\langle g k_{\om,z}^{\eta+2},k_{\om,z}^{\eta+2}\rangle_{L^2_\om}.
$$
Then we define
$$
MO_{\omega,\eta}(f)(z)=\left( B_{\omega,\eta}(|f|^2)(z)-|B_{\omega,\eta}(f)(z)|^2\right)^{\frac12},
$$
which, in some senses, can be treated as a certain global mean oscillations because
\begin{equation}\label{e1}
MO_{\omega,\eta}(f)(z)=\| f k_{\om,z}^{\eta+2}- B_{\omega,\eta}(f)(z) k_{\om,z}^{\eta+2}\|_{A^2_\om},
\end{equation}
and
\begin{equation}\label{e2}
MO_{\omega,\eta}(f)(z)=\left( \int_{\mathbb{D}}\int_{\mathbb{D}} |f(u)-f(\zeta)|^2 |k_{\om,z}^{\eta+2}(u)|^2  |k_{\om,z}^{\eta+2}(\zeta)|^2 \om(u)\om(\z)\,dA(u)\,dA(\zeta) \right)^{\frac12}.
\end{equation}

Before defining the local mean oscillation of a locally square integrable function on $\D$ in the Bergman metric, recall that the Bergman metric $\beta$ on $\D$ is defined by
\begin{equation*}
\beta(z,\zeta)= \dfrac{1}{2} \log \dfrac{1+|\varphi_z(\zeta)|}{1-|\varphi_z(\zeta)|}, \qquad z,w\in \mathbb{D},
\end{equation*}
where $\varphi_z$ is the automorphism of $\mathbb{D}$, i.e. $\varphi_z(\zeta)= \frac{z-\zeta}{1-\overline{z}\zeta}$. For a fixed $r>0$, the Bergman disc $D(z,r)$ centered at $z$ with the radius of $r$ is defined by $D(z,r)=\{\zeta\in \mathbb{D}: \beta(z,\zeta)<r\}.$ It is well-known that $D(z,r)$ is the Euclidean disc centered at $(1-\tanh^2{r})z/(1-\tanh^2{r}|z|^2)$ and of radius $(1-|z|^2)\tanh^2{r}/(1-\tanh^2{r}|z|^2)$. Then the local mean oscillation of $f\in L^2_\om$ in the Bergman metric is defined to be
$$
MO_{\om,r}(f)(z)=\left( \dfrac{1}{\omega(D(z,r))} \int_{D(z,r)} |f(\zeta)-\widehat{f_{\om,r}(z)}|^2 \om(\z)\,dA(\zeta) \right)^{1/2},
$$
where $\om(D(z,r))=\int_{D(z,r)}\om\,dA$ and the averaging function $\widehat{f_{\om,r}}$ is given by
$$
\widehat{f_{\om,r}}(z)= \dfrac{1}{\omega(D(z,r))} \int_{D(z,r)} f(\zeta)\om(\z)\, dA(\zeta).
$$
It is easy to check that for any $z\in\D$ and $r>0$, one has
\begin{equation}\label{e3}
MO_{\om,r}(f)(z)=\left(\dfrac{1}{2\omega(D(z,r))^2} \int_{D(z,r)}\int_{D(z,r)} |f(u)-f(\zeta)|^2\om(u)\om(\z)\, dA(u)\, dA(\zeta) \right)^{\frac12}.
\end{equation}

Let $d\lambda$ be the M\"obius invariant measure on $\D$. That is, $d\lambda(z)=dA(z)/(1-|z|^2)^2$. Our main result can be stated as follows.
\begin{theorem}\label{maintheorem}
    Let $1\leq p<\infty$, $\omega\in\RR$, and $f\in L^2_\om$. Then the following statements are equivalent:
\begin{itemize}
\item[(i)] $H_f$ and $H_{\overline{f}}$ are in $S_p(A^2_\omega,  L^2_\om)$;
\item[(ii)] There exists an $\eta_0=\eta_0(\om)>0$ such that  $MO_{\omega,\eta}(f)\in L^p(\D,~d\lambda)$ for some (equivalently for all) $\eta\geq \eta_0$;
\item[(iii)] $MO_{\om,r}(f)\in  L^p(\D,~d\lambda)$ for some (equivalently for all) $r=r(\om)>0$.
\end{itemize}
\end{theorem}

It is worth mentioning that (iii) in the theorem does actually imply the compactness of both $H_f$ and $H_{\overline{f}}$, provided $\om\in\RR$. Being precise, since $\om\in\RR$, it follows from \cite[(1.2)]{pelaez1} that for any fixed $r>0$ and for any $u, \z\in D(z,r)$,
$$
\om(u)\asymp\om(\z)\asymp\om(z),\quad \om(D(z,r))\asymp\whw(z)(1-|z|)\asymp\om(z)(1-|z|)^2,\quad z\in\D,
$$
and hence \eqref{e3} is comparable to $MO_{2,r}$ in \cite{hu}. Therefore (iii) implies $\lim_{|z|\to1^-} MO_{\om,r}(f)(z)=0$ and hence it follows from \cite[Theorem 4.5]{hu} with the special case $p=q$ that both $H_f$ and $H_{\overline{f}}$ are compact on $A^2_\om$. Alternatively, the similar proof of \cite[Theorem 1]{pelaez4} yields the compactness of both $H_f$ and $H_{\overline{f}}$ directly. Apparently, the result covers partially of \cite[Theorem 1]{pau2} in the case of $1\leq p<\infty$ and in the setting of the unit disc.

The proof of the theorem will be done by verifying that (i) $\Rightarrow$ (iii), (ii) $\Leftrightarrow$ (iii), and (ii) $\Rightarrow$ (i) respectively. The method used to prove that (i) implies (iii) is inherited from \cite{isralowitz}, which depends on an efficient estimate of the local mean oscillation of the symbol $f$. The proof of (ii) $\Leftrightarrow$ (iii) is proved with the aid of Lemma \ref{lemma2}, which can be set up by the approach used in \cite[Lemma 3.1]{pau2}. Meanwhile, we can prove the equivalence between (ii) and (iii) for the full case $0<p<\infty$ and for the involved weight $\om\in\DDD$. To prove that (ii) implies (i), for the full range $0<p<\infty$, instead of using some classical techniques, we construct a linear bounded operator on $A^2_\om$, which makes the proof easier and hence avoid lots of laborious calculations. To be more concrete, suppose $\{e_j\}$ is an orthonormal basis for $A^2_\om$. Note that if $\om\in\DDD$, then for a large enough $\eta=\eta(\om)$
\begin{equation}\label{kernel}
|k_{\om,z}^{\eta+2}(\zeta)|\asymp\frac{(1-|z|)^{\eta+3/2} \widehat{\omega}(z)^{-1/2}}{|1-\overline{z}\zeta|^{\eta+2}},\quad
z,\z\in\D.
\end{equation}
This together with the atomic decomposition (see \cite[Theorem 2]{Pelaez7} with the special case $p=q=2$) enable us to define a certain linear operator $A: A^2_\om\to A^2_\om$ satisfying $A e_j =k_{\om,a_j}^{\eta+2}$ and $A$ is bounded and onto, where $\{a_j\}$ is $r$-lattice of $\D$. Then an application of \cite[Theorem 1.27 and Proposition 1.31]{zhu2} gives everything we aim for.

A careful reader may have already realized that the proof of (i) $\Rightarrow$ (iii) is just dealt with in the case of $\om\in\RR$ instead of $\om\in\DDD$. However, a substantial obstacle
will appear in the proof if one tries to use a similar method for a general weight in $\DDD$. Indeed, if $\om$ is only assumed to belong to $\DDD$, then by \cite[Theorem 1]{pelaez2}, we  may find a large enough $r_0=r_0(\om)$ such that for all $r>r_0$,
\begin{equation}\label{eq:important}
\|B^\om_z\|^2_{A^2_\om}\asymp\om(D(z,r))^{-1}\asymp\frac{1}{\whw(z)(1-|z|)},\quad z\in\D.
\end{equation}
Nevertheless, we are not in a position to obtain the same estimate as in Lemma \ref{lemma5} because $r$ is supposed to be sufficiently small. This obstacle does not happen if $\om\in\RR$, since the last \eqref{eq:important} is valid for all $r>0$. Therefore, some new techniques should be developed in this case.

We finish the introduction by a couple of words about the notation used in this paper. Throughout the paper $\frac1p+\frac{1}{p'}=1$ for $1<p<\infty$. Further, the letter $C=C(\cdot)$ will denote an absolute constant whose value depends on the parameters indicated in the parenthesis, and which may change from one occurrence to another. If there exists a constant
$C=C(\cdot)>0$ such that $a\le Cb$, then it is written either $a\lesssim b$ or $b\gtrsim a$. In particular, if $a\lesssim b$ and
$a\gtrsim b$, then it is denoted by $a\asymp b$ and said that $a$ and $b$ are comparable.

\section{Auxiliary results}
In this section, we are going to present several auxiliary lemmas that are useful for proving Theorem \ref{maintheorem}. Some of them are proved not only for regular weights but for a wider class of weights.

We begin with a simple but important result for $\om\in\DDD$.

\begin{lemma}\label{eq:weight1}
Let $\om\in\DDD$. Then for any $c\geq0$ there exists an $\eta_0=\eta_0(\om,c)>0$ such that for all $\eta>\eta_0$
\begin{equation}
    \int_\D|K^\eta_z(\z)|\beta(z,\z)^c\om(\z)\,dA(\z)\lesssim\frac{\whw(z)}{(1-|z|)^{\eta-1}},\quad z\in\D.
\end{equation}
\end{lemma}
\begin{proof}
    The definition of the Bergman distance implies that $\beta$ grows logarithmically, and hence for any $\varepsilon>0$, we have
    \begin{equation}\label{1}
        \beta(z,\z)\lesssim \left( \dfrac{(1-|z|)(1-|\zeta|)}{|1-\overline{z}\zeta|^2}\right)^{-\varepsilon},\quad z,\z\in\D.
    \end{equation}
Since $\om\in\DDD$, by  \cite[Lemma B]{PelRaSi2018}, we see that there exists a $\beta=\beta(\om)>0$ such that the function $\frac{\whw(r)}{(1-r)^\beta}$ is essentially decreasing on $(0,1)$. For a sufficiently small $\varepsilon=\varepsilon(\om,c)>0$ such that $c\varepsilon<\beta$, we have $\om_{[-c\varepsilon]}(z)=\om(z)(1-|z|)^{-c\varepsilon}\in\DD$. Indeed, on one hand, since $\om\in\DDD\subset\DD$,
$$
\widehat{\om_{[-c\varepsilon]}}(\frac{1+r}{2})=\int_{\frac{1+r}{2}}^1\frac{\om(t)}{(1-t)^{c\varepsilon}}\,dt\geq(1-\frac{1+r}{2})^{-c\varepsilon}\whw(\frac{1+r}{2})\asymp(1-r)^{-c\varepsilon}\whw(r),\quad 0<r<1.
$$
On the other hand,  using an integration by parts, the fact that $\om\in\DDD\subset\Dd$ and \cite[Lemma B]{PelRaSi2018}, we have
\begin{equation}\label{eq1}
    \begin{split}
     \widehat{\om_{[-c\varepsilon]}}(r)&=\int^1_{r}(1-t)^{-c\varepsilon} \omega(t)\,dt=-\int_{r}^1 (1-t)^{-c\varepsilon}\,d\whw(t)\\
     &=\whw(t)(1-t)^{-c\varepsilon}
     \bigg|^r_1+c\varepsilon\int^1_r\frac{\whw(t)}{(1-t)^{c\varepsilon+1}}\,dt\\
     &\lesssim
     \whw(r)(1-r)^{-c\varepsilon}+\frac{\whw(r)}{(1-r)^\beta}\int_r^1\frac{dt}{(1-t)^{c\varepsilon+1-\beta}}\\
     &\asymp\whw(r)(1-r)^{-c\varepsilon},\quad 0<r<1.
    \end{split}
\end{equation}
Therefore, $\om_{[-c\varepsilon]}(z)=\om(z)(1-|z|)^{-c\varepsilon}\in\DD$. Now, for any $\eta>\eta_0=\eta_0(\om,c)$ with  $\eta_0>2c\varepsilon+1$, \eqref{1}, \cite[Lemma 2.1 (vii)]{pelaez5} and \eqref{eq1} yield
\begin{equation*}
\begin{split}
    \int_\mathbb{D} |K_z^\eta(\zeta)| \beta(z,\zeta)^c \omega(\zeta) dA(\zeta)&= \int_\mathbb{D} \dfrac{1}{|1-\overline{z}\zeta|^\eta} \beta(z,\zeta)^c \omega(\zeta) dA(\zeta)\\
&\lesssim (1-|z|)^{-c\varepsilon} \int_\mathbb{D} \dfrac{1}{|1-\overline{z}\zeta|^{\eta-2c\varepsilon}} (1-|\zeta|)^{-c\varepsilon} \omega(\zeta) dA(\zeta)\\
&\asymp (1-|z|)^{c\varepsilon+1-\eta}\int^1_{|z|}(1-r)^{-c\varepsilon} \omega(r)\, dr\\
&\lesssim (1-|z|)^{c\varepsilon+1-\eta} \whw(z)(1-|z|)^{-c\varepsilon}\\
&=\dfrac{\whw(z)}{(1-|z|)^{\eta-1}},\quad z\in\D.
\end{split}
\end{equation*}

\end{proof}

The following lemma plays a critical role in the proof, while the corresponding result can not be generalized to the case $\om\in\DDD$ by the same method.
\begin{lemma}\label{lemma5}
Let  $\omega\in\RR$. Then there exists an $r=r_0(\om)>0$ such that
$$
MO_{\om,r}(f)(z)^2\lesssim \frac{1}{\omega(D(z,r))}\int_{D(z,r)} \left| \int_{D(z,r)}(f(u)-f(\zeta)) B_u^\omega(\zeta)\omega(\zeta)\, dA(\z)\right|^2\omega(u)\,dA(u),\quad z\in\D.
$$
\end{lemma}

\begin{proof}
Using the same method of proving \cite[p.129]{luecking1986}, we see that for any $0<p<\infty$ and $0<\sigma<r$
$$
| f(u)-f(z)|^p \lesssim\frac{\beta(u,z)^p}{(1-|z|)^2}  \int_{D(z,r)} |f(\zeta)|^p \,dA(\zeta),\quad u\in D(z,\sigma),~~f\in\H(\D).
$$
This together with \cite[Lemma 6 and (2.4)]{PelRaSi2018} yields for any fixed $r>0$
\begin{equation*}
    \begin{split}
|B_u^\omega(\zeta)-B_z^\omega(z)|&\leq|B_u^\omega(\zeta)-B_u^\omega(z)|+|B_z^\omega(u)-B_z^\omega(z)|\\
&\lesssim\frac{\beta(\zeta,z)+\beta(z,u)}{(1-|z|)^2}  \int_{D(z,r)} |B_\z^\omega(v)|\,dA(v)\\
&\lesssim \frac{\beta(\zeta,z)+\beta(z,u)}{\om(D(\z,r))}\asymp
(\beta(\zeta,z)+\beta(z,u)) B_z^\omega(z),\quad z\in\D,~~u,\z\in D(z,r),
    \end{split}
\end{equation*}
where the last estimate holds due to the fact that $B^\om_z(z)=\|B^\om_z\|_{A^2_{\om}}^2$. That is, there exists a constant $C=C(\om)$ such that
$$
\left|\frac{B_u^\omega(\zeta)}{\|B^\omega_{z}\|^2} -1\right|\leq C(\beta(\zeta,z)+\beta(z,u))\leq C r,\quad z\in\D.
$$
Now, if we take $L_u(\z)=\left(\frac{B_u^\omega(\zeta)}{\|B^\omega_{z}\|^2} -1\right)$, then $|L_u(\z)|\leq Cr$ for all $u,\zeta\in D(z,r)$. Since $\om\in\RR$, it follows from \eqref{eq:important} that $\|B^\om_z\|^2_{A^2_\om}\asymp\om(D(z,r))^{-1}\asymp\frac{1}{\whw(z)(1-|z|)}$ for all $r>0$.  Therefore, we have
\begin{equation}\label{eq}
    \begin{split}
  &\quad\frac{1}{\omega(D(z,r))}\int_{D(z,r)} \left| \int_{D(z,r)}(f(u)-f(\zeta)) B_u^\omega(\zeta)\omega(\zeta)\, dA(\z)\right|^2\omega(u)\,dA(u)\\
  &= \frac{1}{\omega(D(z,r))}\int_{D(z,r)} \left| \int_{D(z,r)}(f(u)-f(\zeta)) (1+L_u(\zeta))\|B^\omega_{z}\|^2 \omega(\zeta)\, dA(\z)\right|^2\omega(u)\,dA(u)\\
  &\asymp \frac{1}{\omega(D(z,r))^3}\int_{D(z,r)} \left| \int_{D(z,r)}(f(u)-f(\zeta)) (1+L_u(\zeta)) \omega(\zeta)\, dA(\z)\right|^2\omega(u)\,dA(u)
  \end{split}
\end{equation}
Then applying the triangle inequality that $|a+b|^2\geq \frac{|a|^2}{2}-|b|^2$ to the inner integral of the last formular above, we have 
\begin{equation}\label{tri}
  \begin{split}
     &\quad\left| \int_{D(z,r)}(f(u)-f(\zeta)) (1+L_u(\zeta)) \omega(\zeta)\, dA(\z)\right|^2\\
&=\left| \int_{D(z,r)}f(u)-f(\zeta))\om(\z)\,dA(\z)+\int_{D(z,r)} (f(u)-f(\zeta))L_u(\zeta)\omega(\zeta)\,dA(\z)\right|^2\\
&\geq\frac{1}{2}\left|\int_{D(z,r)}f(u)-f(\zeta))\om(\z)\,dA(\z)\right|^2-\left|\int_{D(z,r)} (f(u)-f(\zeta))L_u(\zeta)\omega(\zeta)\,dA(\z)\right|^2\\
&\geq\frac{1}{2}\left|\int_{D(z,r)}f(u)-f(\zeta))\om(\z)\,dA(\z)\right|^2-\left(\int_{D(z,r)} |(f(u)-f(\zeta))L_u(\zeta)|\omega(\zeta)\,dA(\z)\right)^2
    \end{split}
\end{equation}
Therefore, \eqref{eq}, \eqref{tri} and the above estimate of $|L_u(\z)|$ yield
\begin{equation}\label{0}
    \begin{split}
      &\quad\frac{1}{\omega(D(z,r))}\int_{D(z,r)} \left| \int_{D(z,r)}(f(u)-f(\zeta)) B_u^\omega(\zeta)\omega(\zeta)\, dA(\z)\right|^2\omega(u)\,dA(u) \\
  &\geq\frac{1}{2\omega(D(z,r))}\int_{D(z,r)} \left. \left| \frac{1}{\omega(D(z,r))}\int_{D(z,r)}(f(u)-f(\zeta)) \omega(\zeta)\,dA(\zeta)\right|^2\right.\om(u)\,dA(u)\\
  &\quad\left.-\frac{C^2r^2}{\omega(D(z,r))^3} \int_{D(z,r)}\left( \int_{D(z,r)} |f(u)-f(\zeta)| \omega(\zeta)\,dA(\zeta)\right)^2\right.\,\omega(u)\,dA(u). 
    \end{split}
\end{equation}

On one hand, we have 
\begin{equation}\label{11}
    \frac{1}{\omega(D(z,r))}\int_{D(z,r)}(f(u)-f(\zeta)) \omega(\zeta)\,dA(\zeta)=
f(u)-\widehat{f_{\om,r}}(z).
\end{equation}
On the other hand, Cauchy-Schwarz inequality yields
\begin{equation}\label{2}
\begin{split}
 &\quad\int_{D(z,r)}  \left( \int_{D(z,r)} |f(u)-f(\zeta)| \omega(\zeta)\,dA(\zeta)\right)^2 \omega(u)\,dA(u)\\
&\leq\omega(D(z,r)) \int_{D(z,r)} \int_{D(z,r)} |f(u)-f(\zeta)|^2 \omega(\zeta)\omega(u)\,dA(\zeta)\, dA(u)\\
&= \omega(D(z,r))^3 \frac{1}{\omega(D(z,r))^2} \int_{D(z,r)} \int_{D(z,r)} |f(u)-f(\zeta)|^2 \omega(\zeta)\omega(u)\,dA(\zeta)\, dA(u)\\
&=2 \omega(D(z,r))^3 MO_{\om,r}(f)(z)^2,\quad z\in\D.
\end{split}
\end{equation}
Then, combining \eqref{0}, \eqref{11} and \eqref{2}, we deduce
\begin{equation*}
    \begin{split}
     &\quad\frac{1}{\omega(D(z,r))}\int_{D(z,r)} \left| \int_{D(z,r)}(f(u)-f(\zeta)) B_u^\omega(\zeta)\omega(\zeta)\, dA(\z)\right|^2\omega(u)\,dA(u) \\
  &\geq (\frac{1}{2}-{2C^2r^2})MO_{\om,r}(f)(z)^2,\quad z\in\D.  
    \end{split}
\end{equation*}
Finally, by choosing $r$ so that $0<C^2r^2<\frac14$, we arrive at the desired result.
\end{proof}

The following lemma is critical to the proof of the lemma \ref{lemma2}.
\begin{lemma}\label{lemma1}
Let $\omega\in \DDD$ and $f\in L^2_\om$. Then there exist an $r_0=r_0(\om)>0$ such that for any $r\geq r_0$ and $z, \z\in\D$ with $\beta(z,\zeta)<r$  we have
$$
|\widehat{f_{\om,r}}(z)-\widehat{f_{\om,r}}(\zeta)|\lesssim MO_{\om,2r}(f)(z), \qquad\widehat{|g|_r^2}(z)\lesssim MO_{\om,2r}(f)(z)^2,
$$
where $g=f-\widehat{f_{\om,r}}$.
\end{lemma}

\begin{proof}
We first observe that since $\om\in\DDD$ by the hypothesis, there exists an $r_0=r_0(\om)>0$ such that $\om(D(z,r))\asymp\widehat{\om}(z)(1-|z|)$ as $|z|\to1^-$, provided $r>r_0$. This asymptotic equality together with Fubini's theorem and H\"{o}lder inequality indicates that for $z,\zeta\in \mathbb{D}$ with $\beta(z,\zeta)<r$
\begin{align*}
|\widehat{f_{\om,r}}(z)-\widehat{f_{\om,r}}(\zeta)|^2
&=\left|\dfrac{1}{\omega(D(z,r))\omega(D(\zeta,r))} \int_{D(\zeta,r)} \int_{D(z,r)} f(u)\om(u)\, dA(u)\om(v)\,dA(v)\right.\\
 & \ -\left.\dfrac{1}{\omega(D(z,r))\omega(D(\zeta,r))}\int_{D(z,r)} \int_{D(\zeta,r)}  f(v)\om(v)\, dA(v)\om(u)\,dA(u)\right|^2\\
 &\leq\frac{1}{\omega(D(\zeta,r))\omega(D(z,r))} \left(\int_{D(z,r)}\int_{D(\zeta,r)} |f(u)-f(v)| \om(u)\om(v) \,dA(u) \,dA(v)\right)^2\\
&\lesssim\frac{1}{\omega(D(\zeta,r))\omega(D(z,r))} \int_{D(z,r)}\int_{D(\zeta,r)} |f(u)-f(v)|^2 \om(u)\om(v) \,dA(u) \,dA(v)\\
&\lesssim\frac{1}{\omega(D(z,2r))^2} \int_{D(z,2r)}\int_{D(z,2r)} |f(u)-f(v)|^2 \om(u)\om(v) \,dA(u) \,dA(v)\asymp MO_{\om,2r}^2(f)(z).
\end{align*}
Next, to see the second inequality, the triangle inequality and the first inequality yield
\begin{align*}
\left(\widehat{|g|_r^2}(z)\right)^{\frac12}&=\left(\frac{1}{\omega(D(z,r))} \int_{D(z,r)} |f(\zeta)-\widehat{f_{\om,r}}(\zeta)|^2 dA_\omega(\zeta)\right)^{\frac12} \\
&\leq\left(\dfrac{1}{\omega(D(z,r))} \int_{D(z,r)} |f(\zeta)-\widehat{f_{\om,r}}(z)|^2\om(\z)\, dA(\zeta)\right)^{\frac12}\\
&\quad+\left(\dfrac{1}{\omega(D(z,r))} \int_{D(z,r)} |\widehat{f_{\om,r}}(z)-\widehat{f_{\om,r}}(\zeta)|^2\om(\z)\,dA(\zeta)\right)^{\frac12}\\
&\lesssim \left(\frac{1}{\omega(D(z,2r))} \int_{D(z,2r)} |f(\zeta)-\widehat{f_{\om,r}}(z)|^2 \om(\z)\,dA(\zeta)\right)^{\frac12} \\
 &\quad+ MO_{\om,2r}(f)(z) \left( \frac{1}{\omega(D(z,r))} \int_{D(z,r)}  \om(\z)\,dA(\zeta) \right)^{\frac12} \asymp MO_{\om,2r}(f)(z).
\end{align*}
The proof is complete.
\end{proof}

Minor modifications in the proof of \cite[Lemma 3.1]{pau2} together with Lemma~ \ref{lemma1} yield the following result, which plays a key role in the proof of the main theorem.

\begin{lemma}\label{lemma2}
   Let $0<p,d,\delta<\infty$ and $\omega\in \mathcal{D}$. Then there exist an $r_0=r_0(\om)>0$ such that for any $r\geq r_0$ and any for $r$-lattice $\{a_j\}$,
$$
|\widehat{f_{\om,r}}(z)-\widehat{f_{\om,r}}(\z)|\lesssim N_p(f,\z)^{1/p} |1-\overline{\z}z|^d (1+\beta(\z,z))(\text{min}~\{(1-|z|,1-|\z|)\})^{-\delta}\quad z,\z\in \mathbb{D},
$$
with
$$
N_p(f,\z)=\sum_j^\infty
\frac{MO_{\om,2r}(f)(a_j)^p (1-|a_j|^2)^{\delta p}}{|1-\overline{\z}a_j|^{pd}}.
$$
\end{lemma}

\begin{proof}
The argument is similar as the proof of \cite[Lemma ~3.1]{pau2}. For the convenience of the reader, we provide the proof. For any $z, \z\in\D$, let $\gamma(t)$, $0\leq t\leq1$ be the geodesic in the Bergman metric from $z$ to $\z$. Let $N=[{3\beta(z,\z)}/{r}]+1$ (where $[x]$ denotes the largest integer no greater than $x$) and $t_j=j/N$, $j=0, 1,\cdots,N.$ Setting $z_j=\gamma(t_j)$, we have 
$$
\beta(z_j,z_{j+1})=\frac{\beta(z,\z)}{N}\leq \frac{r}{3}.
$$
    For each $j$, choose a point $a_j$ in the lattice such that $\beta(z_j,a_j)\leq \frac{r}{2}$ and hence $\beta(z_{j-1},a_j)\leq \frac{5r}{6}<r$. Then the above estimate, triangle inequality and Lemma \ref{lemma1} yield
    \begin{equation}\label{xxx}
        \begin{split}
    |\widehat{f_{\om,r}}(z)-\widehat{f_{\om,r}}(\z)|&\leq\sum_{j=1}^N|\widehat{f_{\om,r}}(z_{j-1})-\widehat{f_{\om,r}}(z_j)|\\
&\leq\sum_{j=1}^N|\widehat{f_{\om,r}}(z_{j-1})-\widehat{f_{\om,r}}(a_j)|+|\widehat{f_{\om,r}}(z_{j})-\widehat{f_{\om,r}}(a_j)|\\
&\lesssim \sum_{j=1}^N MO_{\om,2r}(f)(a_j), \quad z,\z\in\D.
        \end{split}
    \end{equation}
It follows from the proof of \cite[Lemma ~3.1]{pau2} that 
$$
\frac{|1-\overline{\z}a_j|}{|1-\overline{\z}z|}\leq2, \quad z,\z\in\D.
$$
Applying this estimate into \eqref{xxx}, for any $0<d<\infty$ we deduce
$$
|\widehat{f_{\om,r}}(z)-\widehat{f_{\om,r}}(\z)|\lesssim\sum_{j=1}^N\frac{MO_{\om,2r}(f)(a_j)}{|1-\overline{\z}a_j|^d}|1-\overline{\z}z|^d, \quad z,\z\in\D.
$$
Finally, by repeating the same steps as they were used in the proof of \cite[Lemma ~3.1]{pau2}, we are in a position to get the result we are aiming for.
\end{proof}

To prove the main result, we also need the following result, which is dealt with for a wider class of weights.

\begin{lemma} \label{lemma3}
Let $\omega\in \widehat{\mathcal{D}}$ and $f\in L^2_\om$. Then, there exists an $\eta_0=\eta_0(\om)>0$ such that for all $\eta>\eta_0$
\begin{equation*}
 \|H_f k_{\om,z}^{\eta+2}\|_{A^2_\om}+ \|H_{\overline{f}}k_{\om,z}^{\eta+2}\|_{A^2_\om}\lesssim MO_{\omega,\eta}(f)(z),\quad z\in\D.
\end{equation*}
\end{lemma}
\begin{proof}
Since $\om\in\DD$, by \cite[Lemma 2.1]{pelaez5}, there exists an $\eta_0=\eta_0(\om)>0$ such that for all $\eta>\eta_0$, $k_{\om,z}^{\eta+2}\in A_\omega^2$. Therefore, Cauchy-Schwarz's inequality yields
$$
|B_{\omega,\eta}(f)(z)|=|\langle f k_{\om,z}^{\eta+2}, k_{\om,z}^{\eta+2}\rangle_{A^2_\om}|=|\langle P_\omega(f k_{\om,z}^{\eta+2}),k_{\om,z}^{\eta+2} \rangle_{A^2_\omega}| \leq \| P_\omega(f k_z^{\eta+2})\|_{A^2_\omega} .
$$
This together with the Pythagorean theorem implies
\begin{align*}
\|H_f k_{\om,z}^{\eta+2} \|_{A^2_\omega}&= \left(\|f k_{\om,z}^{\eta+2}\|_\omega^2 - \|P_\omega(f k_{\om,z}^{\eta+2})\|_\omega^2 \right)^{\frac12} \leq  \left(\|f k_{\om,z}^{\eta+2}\|_\omega^2 - |B_{\omega,\eta}(f)(z)|^2\right)^{\frac12}\\
&=\left( B_{\omega,\eta}(|f|^2)(z)-|B_{\omega,\eta}(f)(z)|^2\right)^{\frac12} = MO_{\omega,\eta}(f)(z).
\end{align*}
Likewise, we obtain $\|H_{\overline{f}}k_z^{\eta+2}\|_{A^2_\omega}\lesssim MO_{\omega,\eta}(f)(z)$. Therefore,
$$\|H_f k_{\om,z}^{\eta+2}\|_{A^2_\omega}+ \|H_{\overline{f}} k_{\om,z}^{\eta+2}\|_{A^2_\omega}\lesssim MO_{\omega,\eta}(f)(z).$$
\end{proof}

\section{Proof of Theorem \ref{maintheorem}}
Now, we are ready to prove our main result. The proof will be finished by verifying that (ii) $\Rightarrow$ (i), (i) $\Rightarrow$ (iii) and (ii) $\Leftrightarrow$ (iii) respectively. A sequence $\{\a_j\}$ of distinct points in $\D$ is called $r$-separated if $\inf_{i\neq j}\beta(a_i, a_j)>r>0$. $\{\a_j\}$ is called $r$-lattice if it is $r$-separated and satisfies $\D=\cup_{j}D(a_j,r)$. Here and from now on, write simply $D(a_j,r)$ by $D_j$.

The method applied here to prove (ii) $\Rightarrow$ (i) is quite different from the approaches used to prove the case of the standard weight, which relies on a construction of a linear bounded operator on $A^2_\om$ and hence avoids lots of tedious calculations. Moreover, we will also prove (ii) $\Rightarrow$ (i) for the full range $0<p<\infty$.
\vskip0.3cm

\noindent
{\bf Proof of (ii) $\Rightarrow$ (i)}: Assume (ii) Holds. Again, since $\om\in\DDD$, there exists an $\eta_0=\eta_0(\om)>0$ such that for all $\eta>\eta_0$, $k_z^{\eta+2}\in A_\omega^2$. Then using the atomic decomposition of $A^2_\om$ (see \cite[Theorem 2]{Pelaez7} with the special case $p=q=2$), we see that for any $f\in A^2_\om$ there exists a sequence $\{\lambda\}_j\in \ell^2$ such that 
\begin{equation}\label{atom}
f(z)=\sum_{j=0}^\infty \lambda_j k^{\eta+2}_{\om,a_j}(z), \quad z\in\D.
\end{equation}
for a $r=r(\om)$-lattice $\{\a_j\}$. Conversely, each $f$ with the form of \eqref{atom} and $\{\lambda\}_j\in\ell^2$ must belong to $A^2_\om$. Those facts enable us to define a linear bounded surjective operator on $A^2_\om$, which plays a critical role in the proof. To be concrete, let $\{e_j\}_j$ be an orthonormal basis of $A^2_\om$. Then for each $f\in A^2_\om$, it can be written as  
$$
f(z)=\sum_{j=0}^\infty \langle f, e_j\rangle_{A^2_\om}\,e_j(z), \quad z\in\D,
$$
where $\langle \cdot, \cdot\rangle_{A^2_\om}$ is the inner product of $A^2_\om$. Moreover, we have $\|\{\langle f, e_j\rangle_{A^2_\om}\}_j\|_{\ell^2}=\|f\|_{A^2_\om}$. Now, define the linear operator $A: A^2_\om\to A^2_\om$ as follows:
\begin{align*}
\begin{split}
A(f)(z)=\sum_{j=0}^\infty \langle f, e_j\rangle_{A^2_\om}\, k^{\eta+2}_{\om,a_j}(z), \quad z\in\D.
 \end{split}
 \end{align*}
It is easy to see that $A$ is bounded and onto and 
 $$
 Ae_j=k^{\eta+2}_{\om,a_j}, \quad j=0,1,\cdots.
 $$
 We will see that (ii) implies (iii) trivially and hence both $H_f$ and $H_{\overline{f}}$ are compact on $A^2_\om$ by \cite[Theorem 4.5]{hu}. Therefore, to show (i), it suffices to show that both
$A^*H_f^*H_f A,$ ~$A^* H^*_{\overline{f}} H_{\overline{f}} A \in S_\frac{p}{2}$
by \cite[Proposition 1.30]{zhu2}, which can be done if we show
$$
\sum_j \langle (A^* H_f^*H_f A)^\frac{p}{2} e_j,e_j\rangle_{A^2_\om} +\langle (A^* H^*_{\overline{f}} H_{\overline{f}} A)^\frac{p}{2} e_j,e_j\rangle_{A^2_\om}<\infty.
$$
If $0< p\leq2$, then by \cite[Proposition 1.31]{zhu2}, Lemma~\ref{lemma3}, and the hypothesis that $\om\in\DDD$, we deduce
\begin{equation}\label{x}
\begin{split}
    &\quad\sum_j \langle (A^* H_f^*H_f A)^\frac{p}{2} e_j,e_j\rangle_{A^2_\om} +\langle (A^* H^*_{\overline{f}} H_{\overline{f}} A)^\frac{p}{2} e_j,e_j\rangle_{A^2_\om}\\
    &\leq\sum_j \langle A^* H_f^*H_f A e_j,e_j\rangle_{A^2_\om}^\frac{p}{2} +\langle A^* H^*_{\overline{f}} H_{\overline{f}} A e_j,e_j\rangle_{A^2_\om}^\frac{p}{2}\\
    &=\sum_j \|H_f k_{\om,a_j}^{\eta+2}\|_{L^2_\om}^p+ \|H_{\overline{f}}k_{\om,a_j}^{\eta+2}\|_{L^2_\om}^p\\
    &\lesssim\sum_j MO_{\om,\eta}(f)(a_j)^p\asymp\sum_j\int_{D_j}MO_{\om,\eta}(f)(a_j)^p\,d\lambda(z).
    \end{split}
\end{equation}
Now, it follows from \cite[(2.20)]{zhu1} that $|1-\overline{a_j}\z|\asymp|1-\overline{z}\z|$ for all $\z\in\D$ and $z\in D_j$. This estimate together with \eqref{kernel} and the fact that $\whw(a_j)\asymp\whw(z)$ as long as $z\in D_j$, we get that there exists an $\eta_1=\eta_1(\om)>\eta_0$ such that for all $\eta\geq \eta_1$,
$$
k^{\eta+2}_{\om,a_j}(\z)\asymp k^{\eta+2}_{\om,z}(\z),\quad z\in D_j,~~\z\in\D.
$$

Applying this estimate into \eqref{x} and noticing the definition of $MO_{\om,\eta}$, we deduce
\begin{equation*}
    \begin{split}
       &\quad\sum_j \langle (A^* H_f^*H_f A)^\frac{p}{2} e_j,e_j\rangle_{A^2_\om} +\langle (A^* H^*_{\overline{f}} H_{\overline{f}} A)^\frac{p}{2} e_j,e_j\rangle_{A^2_\om}\\
&\lesssim\sum_j\int_{D_j}MO_{\om,\eta}(f)(a_j)^p\,d\lambda(z)\asymp\sum_j\int_{D_j}MO_{\om,\eta}(f)(z)^p\,d\lambda(z) \\
    &\lesssim \int_{\mathbb{D}} MO_{\omega,\eta}(f)(z)^p \,d{\lambda(z)}=\| MO_{\omega,\eta}(f)\|_{L^p(\D, d\lambda)}^p<\infty,
    \end{split}
\end{equation*}
which is the desired result we are aiming for.

If $2<p<\infty$, then we may reach the same result by following the above steps by applying \cite[Theorem 1.27]{zhu2} instead of \cite[Proposition 1.31]{zhu2}.\hfill$\Box$

We are in a position to prove that (ii) $\Leftrightarrow$ (iii) for the full range $0<p<\infty$ and for the weight belongs to $\DDD$. \vskip0.3cm

\noindent
{\bf Proof of (ii) $\Leftrightarrow$ (iii)}: Since $\om\in\DDD$, by the double integration representations \eqref{e2} and \eqref{e3} of $MO_{\om,\eta}$ and $MO_{\om,r}$, we may easily see that
$MO_{\om,r}(f)(z)\lesssim MO_{\omega,\eta}(f)(z)$
for a sufficiently large $r>0$ and $\eta>0$ depending on $\om$, which gives (ii)$\Rightarrow$(iii).

Conversely, suppose (iii) holds. We will prove that (ii) is also valid by applying the technique used in \cite{pau2}. Since $\om\in\DDD$ by the hypothesis, $\om(D(z,r))\asymp\widehat{\om}(z)(1-|z|)$ for a sufficiently large $r=r(w)$, say $r>r_0=r_0(\om)$. Let now $0<p<\infty$ and $\{a_j\}$ be an $r$-lattice of $\D$. Then the fact that the number of discs $D_j$ to which each $z$ may belong is uniformly bounded yields
\begin{equation}\label{aim1}
   \sum_j MO_{\om,r}(f)(a_j)^p \asymp \int_\mathbb{D} MO_{\om,r}(f)(z)^p\,  d\lambda(z).
\end{equation}
 Then for a sufficiently large $\eta$ depending on $\om$, say $\eta>\eta_0=\eta_0(\om)$, by \eqref{e3} and triangle inequality, we deduce
\begin{equation}\label{eqx}
\begin{split}
    MO_{\omega,\eta}(f)(z)^2&\leq\sum_{j}\sum_k\int_{D_j}\int_{D_k} |f(u)-f(\zeta)|^2 |k_{\om,z}^{\eta+2}(u)|^2  |k_{\om,z}^{\eta+2}(\zeta)|^2 \omega(u)\omega(\zeta)\,dA(u) dA(\z)\\
    &\lesssim \sum_{j} |k_{\om,z}^{\eta+2}(a_j)|^2 \int_{D_j} |f(u)-\widehat{f_{\om,r}}(z)|^2\om(u)\,dA(u)\lesssim A_1(f,z)+A_2(f,z),\quad z\in\D,
\end{split}
\end{equation}
where
$$
A_1(f,z)= \sum_{j} MO_{\om,r}(f)(a_j)^2  \omega(D_j) |k_{\om,z}^{\eta+2}(a_j)|^2
$$
and
$$A_2(f,z)= \sum_{j} |k_{\om,z}^{\eta+2}(a_j)|^2 |\widehat{f_{\om,r}}(a_j)-\widehat{f_{\om,r}}(z)|^2 \omega(D_j).$$

Therefore, to show (ii), by \eqref{aim1} and \eqref{eqx}, it suffices to show that
\begin{equation}\label{aim}
    \int_\mathbb{D} A_i(f,z)^{\frac{p}{2}} dA\lambda(z) \lesssim \sum_{j} MO_{\om,r}(f)(a_j)^p, \quad i=1, 2.
\end{equation}

First, let us estimate the above integral with integrant $A_1(f,z)$. Since $\om\in\DDD$, there exists an $\eta_1=\eta_1(\om)>\eta_0$ such that for all $\eta>\eta_1$, $\frac{(1-|\cdot|^2)^{p(\eta+3/2)-2}}{\widehat{\omega}(\cdot)^{p/2}}\in\DDD$.

If $0<p\leq2$, then \cite[Theorem 1]{pelaez2} yields
\begin{equation}\label{eqA1}
\begin{split}
    \int_\mathbb{D} A_1(f,z)^{p/2} dA\lambda(z) &\lesssim \sum_{j} MO_{\om,r}(f)(a_j)^p \omega(D_j)^{\frac{p}{2}} \int_\mathbb{D} \frac{\widehat{\omega}(z)^{-p/2} (1-|z|^2)^{p(\eta+3/2)} }{|1-\overline{a_j}z|^{p(\eta+2)}}\,d\lambda(z)\\
    &\lesssim \sum_{j} MO_{\om,r}(f)(a_j)^p \omega(D_j)^{\frac{p}{2}}  \int_\mathbb{D}\frac{(1-|z|^2)^{p(\eta+3/2)-2}}{|1-\overline{a_j}z|^{p(\eta+2)}\widehat{\omega}(z)^{p/2}}\,dA(z)\\
    &\lesssim \sum_{j} MO_{\om,r}(f)(a_j)^p \omega(D_j)^{\frac{p}{2}}\left(\int_0^{|a_j|} \dfrac{(1-s)^{p(\eta+\frac{3}{2})-1}}{(1-s)^{p(\eta+2)}\widehat{\omega}(s)^{\frac{p}{2}}}+1\right)\\
    &\asymp \sum_{j} MO_{\om,r}(f)(a_j)^p.
\end{split}
\end{equation}

If $2<p<\infty$, then for a sufficiently small $\varepsilon=\varepsilon(\om)>0$, \cite[(14)]{LiuRattya} yields $\om(D_j)\lesssim \whw(a_j)(1-|a_j|)$ and hence  H\"older's inequality, \cite[Lemma 10]{zhu2}, Fubini's theorem and \cite[Theorem 1]{pelaez2} yield
\begin{equation}\label{eqA11}
    \begin{split}
      \int_\mathbb{D} A_1(f,z)^{{\frac{p}{2}} } dA\lambda(z) &\asymp\int_\D \left(\sum_{j} MO_{\om,r}(f)(a_j)^2 \omega(D_j) \frac{\widehat{\omega}(z)^{-1} (1-|z|^2)^{2(\eta+3/2)}}{|1-\overline{a_j}z|^{2\eta+4}}\right)^\frac{p}{2}\,d\lambda(z)\\
      &\lesssim\int_\D \left(\sum_{j} MO_{\om,r}(f)(a_j)^p  \frac{\widehat{\omega}(z)^{-p/2} (1-|z|^2)^{p(\eta+3/2)}}{|1-\overline{a_j}z|^{\frac{p}{2}(2\eta+3-\varepsilon)}} \widehat{\omega}(a_j)^{\frac{p}{2}}\right)\\
      &\quad\cdot\left( \sum_{j} \frac{(1-|a_j|)^{\frac{p}{p-2}}}{|1-\overline{a_j}z|^{(1+\varepsilon)\frac{p}{p-2}}} \right)^{\frac{p-2}{2}}\,d\lambda(z)\\
      &\lesssim \int_\D \sum_{j} MO_{\om,r}(f)(a_j)^p  \frac{\widehat{\omega}(z)^{-p/2} (1-|z|^2)^{\frac{p}{2}(2\eta+3-\varepsilon)}}{|1-\overline{a_j}z|^{\frac{p}{2}(2\eta+3-\varepsilon)}} \widehat{\omega}(a_j)^{\frac{p}{2}}\,d\lambda(z)\\
      &\lesssim \sum_{j} MO_{\om,r}(f)(a_j)^p\widehat{\omega}(a_j)^{\frac{p}{2}}\sum_i \int_{D_i}\frac{\widehat{\omega}(z)^{-p/2} (1-|z|^2)^{\frac{p}{2}(2\eta+3-\varepsilon)}}{|1-\overline{a_j}z|^{\frac{p}{2}(2\eta+3-\varepsilon)}}\,d\lambda(z)\\
      &\lesssim\sum_{j} MO_{\om,r}(f)(a_j)^p\widehat{\omega}(a_j)^{\frac{p}{2}}\int_\mathbb{D} \frac{(1-|z|^2)^{\frac{p}{2}(2\eta+3-\varepsilon)-2}}{|1-\overline{a_j}z|^{{\frac{p}{2}(2\eta+3-\varepsilon)}}} \widehat{\omega}(z)^{-p/2} \, dA(z)\\
      &\lesssim \sum_{j} MO_{\om,r}(f)(a_j)^p.
    \end{split}
\end{equation}

Next, we proceed to estimate the second integral in \eqref{aim}. Lemma \ref{eq:weight1} and Lemma \ref{lemma2} yield
\begin{equation}\label{eqA2}
    \begin{split}
        &\quad\int_\D A_2(f,z)^{\frac{p}{2}}\,d\lambda(z)\asymp\int_{\D}\left(\sum_{j} |\widehat{f_{\om,r}}(a_j)-\widehat{f_{\om,r}}(z)|^2  |k_z^{\eta+2}(a_j)|^2\widehat{\omega}(a_j)(1-|a_j|^2)\right)^{\frac{p}{2}}\,d\lambda(z)\\
        &\lesssim\int_{\D}\left(\widehat{\omega}(z)^{-1} (1-|z|^2)^{2\eta+3-2\delta} N_p(f,z)^{{\frac{p}{2}}} \int_\mathbb{D} \frac{  (1+\beta(z,\zeta))^2 }{|1-\overline{z}\zeta|^{2(\eta-d+2)}} \dfrac{\widehat{\omega}(\zeta)}{1-|\zeta|} dA(\zeta)\right)^{\frac{p}{2}}\,d\lambda(z)\\
        &\lesssim\int_{\D}\left((1-|z|^2)^{2d-2\delta} N_p(f,z)^{2/p}\right)^{\frac{p}{2}}\,d\lambda(z)\\
        &\lesssim \sum_j MO_{\om,r}(f)(a_j)^p (1-|a_j|^2)^{\delta p} \int_\mathbb{D} \dfrac{(1-|z|^2)^{p(d-\delta)-2} }{|1-\overline{a_j}z|^{pd}} \,dA(z)\lesssim \sum_j MO_{\om,r}(f)(a_j)^p.
    \end{split}
\end{equation}
Finally, combining \eqref{eqA1}, \eqref{eqA11}, and \eqref{eqA2}, we prove that (ii) holds.\hfill$\Box$

To finish the proof of the main theorem, it remains to prove that (i) implies (iii). The method used here originates from a technical construction in \cite{isralowitz}. Before presenting the proof, let us recall the definition of the commutator on $L^2_\om.$ For an $f\in L^2_\om$, the commutator $[M_f, P_\om]:=M_f P_\om-P_\om M_f$. It is well-known that the study of $[M_f, P_\om]$ is equivalent to the simultaneous study of $H_f$ and $H_{\overline{f}}$, which can be partially explained by the identity
\begin{equation}\label{commutator}
    [M_f, P_\om]=H_f P_\om-(H_{\overline{f}} P_\om)^*.
\end{equation}

\vskip 3mm
\noindent
{\bf Proof of (i) $\Rightarrow$ (iii)}: Let $1\leq p<\infty$. Suppose $H_f$ and $H_{\overline{f}}$ are both in $S_p$.
Then the identity \eqref{commutator} is also in $S_p$ of $L^2_\om$ due to the boundedness of $P_\om$ on $L^2_\om$. Let now $\{e_j\}_{j=1}^\infty$ be an orthonormal basis for $L^2_\omega$, and let $\{a_j\}$ be a $r$-lattice of $\D$ for a certain $r=r(\om)>0$. Set
$$
h_j(z)=\omega(D_j)^{-1/2}\chi_{D_j}(z)\quad\text{and}\quad g_j(z)=\chi_{D_j}(z) [M_f,P_\omega]h_j(z)/ \|\chi_{D_j} [M_f,P_\omega]h_j\|_{L^2_\om}.
$$
Then it is easy to see that the following two linear operators
$$A(e_j)(z)=g_j(z)\quad \text{and}\quad  B(e_j)(z)=h_j(z),\quad j=1,2,\cdots,~~~z\in\D.$$
are bounded on $L^2_\om$. Therefore for an $f\in L^2_\om$,  $T:=A^*[M_f, P_\om] B\in S_p(L^2_\om)$ and furthermore
\begin{equation}\label{T}
\|T\|^p_{S_p}\lesssim\|[M_f, P_\om]\|^p_{S_p},\quad f\in L^2_\om.
\end{equation}
Moreover, a simple calculation gives
\begin{align*}
    \langle Te_j,e_j\rangle_{L^2_\om}&=\langle  [M_f,P_\omega]h_j,g_j\rangle_{L^2_\om}\\
 &=\int_\mathbb{D}  ([M_f,P_\omega]h_j)(z) \overline{g_j(z)}\om(z) dA(z)\\
&=\frac{1}{\|\chi_{D_j} [M_f,P_\omega]h_j\|_{L^2_\om}}\int_{D_j}  ([M_f,P_\omega]h_j)(z) \overline{ ([M_f,P_\omega]h_j)(z)}\om(z)\, dA(z)\\
&=\|\chi_{D_j} [M_f,P_\omega]h_j\|_{L^2_\om},\quad j=1,2,\cdots.
\end{align*}
This means that $T$ is a positive operator on $L^2_\om$, and hence $T^p\in S_1$ with $\|T\|^p_{S_p}=\|T^p\|_{S_1}$
by \cite[Lemma 1.25]{zhu2}. This together with \cite[Theorem 1.27 and Proposition 1.31]{zhu2} and \eqref{T} yields
$$
\sum_j \langle  [M_f,P_\omega]h_j,g_j\rangle_{L^2_\om}^p=\sum_j \langle Te_j,e_j\rangle_{L^2_\om}^p\leq \sum_j \langle T^p e_j,e_j\rangle_{L^2_\om} <\infty.
$$
Nevertheless, it is elementary to deduce
$$
\|\chi_{D_j} [M_f,P_\omega]h_j\|_{L^2_\om}=\left( \int_{D_j} \left| \int_{D_j} \frac{f(z)-f(\zeta)}{\omega(D_j)^{1/2}} B^\omega_z(\zeta)\om(\z)dA(\zeta)\right|^2 \om(z)\,dA(z)\right)^{\frac12},
$$
which together with Lemma~\ref{lemma5} implies that
$\sum_{j=1}^\infty MO_{\om,r}(f)(a_j)^p<\infty,$ and hence the assertion follows due to \eqref{aim1}. \hfill$\Box$

\end{document}